\documentclass[11pt]{amsart}

\usepackage{amsfonts, amssymb, amscd}
\usepackage{verbatim}
\usepackage{amssymb}
\usepackage{mathrsfs}
\usepackage{graphicx}
\usepackage[nospace,noadjust]{cite}
\usepackage[all]{xy}
\usepackage{tikz}
\usepackage{subfiles}
\usepackage[toc,page]{appendix}

\usepackage{hyperref}

\newcommand{\Zz}{\mathbb{Z}}

\newcommand{\Pp}{\mathbb{P}}
\newcommand{\Rr}{\mathbb{R}}
\newcommand{\Qq}{\mathbb{Q}}

\newcommand{\Nklt}{\operatorname{Nklt}}

\newcommand{\red}{\operatorname{red}}

\newcommand{\Supp}{\operatorname{Supp}}
\newcommand{\Diff}{\operatorname{Diff}}

\newtheorem{theorem}{Theorem}[section]
\newtheorem{lemma}[theorem]{Lemma}

\newtheorem{definition}[theorem]{Definition}

\newtheorem{corollary}[theorem]{Corollary}

\newtheorem{conjecture}[theorem]{Conjecture}

\begin{document}

\title{On a connectedness principle of Shokurov-Koll\'{a}r type}


\begin{abstract}
Let $(X,\Delta)$ be a log pair over $S$, such that $-(K_X+\Delta)$ is nef over $S$. It is conjectured that  the intersection of the non-klt (non Kawamata log terminal) locus of $(X,\Delta)$ with any fiber $X_s$ has at most two connected components. We prove this conjecture in dimension $\leq 4$ and in arbitrary dimension assuming the termination of klt flips.
\end{abstract}
\author{Christopher D. Hacon}
\address{Department of Mathematics, University of Utah, 155 South 1400 East, JWB 301, Salt Lake City, Utah 84112, USA}
\email{hacon@math.utah.edu}

\author{Jingjun Han}
\address{Beijing International Center for Mathematical Research, Peking University, Beijing 100871, China}
\address{Current address: Department of Mathematics, Johns Hopkins University, Baltimore, MD {\rm 21218}, USA}
\email{hanjingjun@pku.edu.cn}

\maketitle


\section{Introduction}\label{sec: introduction}

We work over an algebraically closed field $k$ of characteristic zero. 

Let $\pi: X\to S$ be a proper morphism with connected fibers, and $(X,\Delta)$ a log pair. It is quite useful to understand the connected components of $\Nklt(X,\Delta)$, the non-klt locus of $(X,\Delta)$, and more generally, the connected components of the fibers of $\pi:\Nklt(X,\Delta)\to S$. When $-(K_X+\Delta)$ is $\pi$-nef and $\pi$-big, by the Shokurov-Koll\'{a}r connectedness lemma, any fiber of $\pi|_{\Nklt(X,\Delta)}:\Nklt(X,\Delta)\to S$ is connected. The connectedness lemma has many applications, including the proof of inversion of adjunction.

When $-(K_X+\Delta)$ is $\pi$-nef, the fibers of $\pi|_{\Nklt(X,\Delta)}:\Nklt(X,\Delta)\to S$ are not necessarily connected. This is the case for $X=\Pp ^1$ and $\Delta$  the union of two distinct points. We expect the following conjecture to be true.

\begin{conjecture}\label{conj:1}
	Let $(X,\Delta)$ be a log pair, and $\pi: X\to S$ a proper morphism with connected fibers, such that $-(K_X+\Delta)$ is $\pi$-nef. For any point $s\in S$, $\pi^{-1}(s)\cap \Nklt(X,\Delta)$ has at most two connected components.
\end{conjecture}
In dimension two, Conjecture \ref{conj:1} was proven by Prokhorov in \cite{Prokhorov01}. Assuming the termination of klt flips, Fujino showed Conjecture \ref{conj:1} for the case when $(X,\Delta)$ is dlt and $K_X+\Delta$ is numerically trivial, \cite{Fujino00}. In particular, he proved Conjecture \ref{conj:1} when $\dim X\le 3$. Koll\'{a}r and Kov\'{a}cs showed Conjecture \ref{conj:1} for the case that $K_X+\Delta$ is numerically trivial without any assumption on the termination of flips. They then applied this result in their study of lc (log canonical) singularities, \cite{kk10,KollarMMP13}.

When $(X,\Delta)$ is dlt (divisorial log terminal), one may wonder if the nefness of $-(K_X+\Delta)$ implies the semiampleness of $-(K_X+\Delta)$, and hence Conjecture \ref{conj:1}  would follow from the result in \cite{kk10}. However, this is not the case. Gongyo constructed an example such that $\dim X\ge3$, $(X,\Delta )$ is plt (purely log terminal), $-(K_X+\Delta)$ is big and nef, and $-(K_X+\Delta)$ is not semiample, \cite{Gongyo12}.

In this paper, assuming the termination of klt flips, we prove Conjecture \ref{conj:1}.
\begin{theorem}\label{thm:shoconn}
	Assume the termination of klt flips or that $\dim X\leq 4$. Let $(X,\Delta)$ be a log pair, and $\pi: X\to S$ a proper morphism, such that $-(K_X+\Delta)$ is $\pi$-nef. Fix a point $s\in S$, and assume that $\pi^{-1}(s)$ is connected but $\pi^{-1}(s)\cap \Nklt(X,\Delta)$ is disconnected (as $k(s)$-schemes). Then,
	
	\begin{enumerate}
		
		\item $(X,\Delta)$ is plt in a neighborhood of $\pi^{-1}(s)$, and
		\item $\pi^{-1}(s)\cap\Nklt(X,\Delta)$ has 2 connected components. 
	\end{enumerate}
\end{theorem}
In fact, we will prove a slightly stronger result in Theorem \ref{thm:shoconnstrong}.  In a private communication, M\textsuperscript{c}Kernan and Xu informed us that they had obtained a similar result in an unpublished work.

Here we give a rough idea of the proof. After taking a dlt modification, we may assume $(X,B+\red C)$ is dlt, where $B=\{\Delta\},C=\lfloor \Delta\rfloor$. We run a special kind of $(K_X+B)$-MMP and reach a Mori-fiber space $X\dashrightarrow Y/Z$. This MMP preserves the connected components of the non-klt locus. The theorem then follows by applying the Shokurov-Koll\'{a}r connectedness lemma and the result in  \cite{KollarMMP13}. 
When $\dim X\le 4$, we show that a certain klt MMP terminates by applying the main result of \cite{AHK07}.
We emphasize that our approach is closely based on the ideas of \cite{sh93} and \cite{kk10}, the main new input is the choice of the special MMP.

The proof of Theorem \ref{thm:shoconn} also implies the following corollary, which can be viewed as a stronger version of the connectedness lemma.
\begin{corollary}\label{cor:shoconn}
	Assume the termination of klt flips or that $\dim X\leq 4$. Let $(X,\Delta)$ be a log pair such that $(X,B+\red C)$ is dlt, where $B=\{\Delta\}$, and $C=\lfloor \Delta\rfloor$. Let $\pi: X\to S$ be a proper morphism such that $-(K_X+\Delta)$ is $\pi$-nef, and there exists $\epsilon>0$ such that $-(K_X+B+(1+\epsilon)C)$ is $\pi$-nef. Then, $\Nklt(X,\Delta)$ is connected in a neighbourhood of any fiber of $\pi$.
\end{corollary}

Following the ideas in \cite{KollarMMP13}, as an application of Theorem \ref{thm:shoconn}, we see that the following strengthening of Theorem 4.40 in \cite{KollarMMP13} holds (see also Theorem 1.7 in \cite{kk10}).
\begin{theorem}\label{thm:soucelccenter}
	Assume the termination of klt flips or that $\dim X\leq 4$. Let $\pi:(X,\Delta)\to S$ be a proper morphism such that $-(K_X+\Delta)$ is $\pi$-nef, and $(X,\Delta)$ is dlt. Let $s\in S$ be a point such that $\pi^{-1}(s)$ is connected. Let $Z\subset X$ be minimal (with respect to inclusion) among the lc centers of $(X,\Delta)$ such that $s\in \pi(Z)$. Let $W\subset X$ be an lc center of $(X,\Delta)$ such that $s\in \pi(W)$. Then, there is an lc center $Z_W\subset W$ such that $Z$ and $Z_W$ are $\Pp^1$-linked. In particular, all the minimal (with respect to inclusion) lc centers $Z_i\subset X$ such that $s\in \pi(Z_i)$ are $\Pp^1$-linked to each other.
\end{theorem}

Let $\pi:(X,\Delta)\to S$ be a proper morphism, and recall that we say $W\subset S$ is an lc center (of $(X,\Delta)$ with respect to $\pi$), if $W=\pi(Z)$ for some lc center $Z\subset X$ of $(X,\Delta)$.

An immediate consequence of Theorem \ref{thm:soucelccenter} is the following strengthening of Corollary 4.41 in \cite{KollarMMP13}.
\begin{corollary}\label{cor:soucelccenter}
	Assume the termination of klt flips or that $\dim X\leq 4$. Let $\pi:(X,\Delta)\to S$ be a proper morphism such that $-(K_X+\Delta)$ is $\pi$-nef and $(X,\Delta)$ is dlt. Then,
	\begin{enumerate}
		
		\item every point $s\in S$ is contained in a unique smallest (with respect to inclusion) lc center $W_s\subset S$, and
		\item any intersection of lc centers is also a union of lc centers.
		
	\end{enumerate}
\end{corollary}

\textbf{Acknowledgements}. Much of this work was done when the second author visited the first author at the University of Utah, the second author would like to thank the University of Utah for its hospitality. The first author was partially supported by NSF research grants no: DMS-1300750, DMS-1265285 and by a grant from the Simons Foundation; Award Number: 256202. He would also like to thank the Department of Mathematics and the Research Institute for Mathematical Sciences located in Kyoto University. The authors would like to thank Chen Jiang for providing useful comments on previous drafts. The second author would like to thank his advisors Gang Tian and Chenyang Xu for constant support and encouragement.

\section{Preliminaries}
\subsection{Notations and Conventions.} We refer the reader to \cite{KollarMMP13,KM98} for the basic definitions of log discrepancies and log canonical (lc), divisorially log terminal (dlt), and kawamata log terminal (klt) singularities. Recall that for a log sub pair $(X,\Delta)$, the non-klt locus is defined by
$$\Nklt(X,\Delta):=\{x|(X,\Delta)\text{ is not klt at }x\}.$$


For the reader's convenience, we state the following lemma which is known as the existence of dlt modifications. The proof is due to the first author. 
\begin{lemma}[Existence of dlt modifications, c.f. \cite{HMX14} Proposition 3.3.1]\label{lem:dlt}
	Let $(X,\Delta)$ be a log pair. There is a proper birational morphism $f:X'\to X$, such that
	\begin{enumerate}
		\item $X'$ is $\Qq$-factorial,
		\item $K_{X'}+\Delta'=f^{*}(K_X+\Delta)$, and
		\item $(X',B'+\red C')$ is dlt, where $B'=\{\Delta'\}$ and $C'=\lfloor \Delta'\rfloor$.
	\end{enumerate}
\end{lemma}

Recall the following simple consequence of the Shokurov-Koll\'{a}r connectedness lemma.
\begin{lemma}\label{lem:crepant}
	Let $(X,\Delta)$ be a log pair, and $f:X'\to X$ a proper birational morphism, such that $(X',\Delta')$ is a crepant pullback of $(X,\Delta)$, i.e., $K_{X'}+\Delta'=f^{*}(K_X+\Delta)$. Then there is a bijection between the connected components of $\Nklt(X,\Delta)$ and the connected components of $\Nklt(X',\Delta')$.
\end{lemma}
\begin{proof}
	By the Shokurov-Koll\'{a}r connectedness lemma, the fibers of
	$$f|_{\Nklt(X',\Delta')}:\Nklt(X',\Delta')\to \Nklt(X,\Delta)$$
	are connected, thus the preimage of a connected component of $\Nklt(X,\Delta)$ is a connected component of $\Nklt(X',\Delta')$.
\end{proof}

\begin{lemma}\label{lem:extremal}
	Let $(X,B)$ be a klt pair, and $-(K_X+B)=C+P$ for an effective divisor $C$ and a nef divisor $P$. Then there exists $\beta>0$, such that
	for any extremal ray $R$ such that $(K_X+B+\beta P)\cdot R<0$, we have $P\cdot R=0$.
\end{lemma}
\begin{proof}
	Pick any $\beta>0$, and let $R_i$ be extremal rays such that $(K_X+B+\beta P)\cdot R_i<0$. Then $(K_X+B)\cdot R_i<0$. Let $\Gamma_i$ be a curve generating $R_i$ having the minimal degree with respect to an ample divisor $H$. In some references, $\Gamma_i$ is called an extremal curve. 
	There are positive real numbers $r_1,\ldots,r_s$ and a natural number $m$ such that
	$$-(C+P)\cdot\Gamma_i=(K_X+B)\cdot \Gamma_i=\sum_{j=1}^s \frac{r_jn_{i,j}}{m},$$
	where $-2(\dim X) m\le n_{i,j}\in\Zz$ (cf.  \cite[Lemma 3.1]{Birkar10}). We deduce that there are only finitely many possibilities for the $n_{i,j}$ and hence for the intersection numbers $(K_X+B)\cdot \Gamma_i$.
	
	Choose $\epsilon>0$, such that $(X,B+\epsilon C)$ is klt. We have
	\begin{align*}
	(K_X+B+\epsilon C)\cdot \Gamma_i=&-(P+(1-\epsilon)C)\cdot \Gamma_i\\
	=&-\epsilon P\cdot \Gamma_i-(1-\epsilon) (P+C)\cdot \Gamma_i<0.
	\end{align*}
	By the same argument, the possibilities for the intersection numbers $(K_X+B+\epsilon C)\cdot \Gamma_i$ are also finite, hence the same holds for $P\cdot \Gamma_i$. We may assume that $P\cdot \Gamma_i\ge \alpha$, if $P\cdot \Gamma_i\neq 0$.
	
	We get the desired $\beta$ by choosing $\beta=\frac{2\dim X}{\alpha}+1$.
\end{proof}

\begin{definition}[{\cite[Definition 4.36]{KollarMMP13}}]
	A standard $\Pp^1$-link is a $\Qq$-factorial pair $(X,D_1+D_2+\Delta)$ plus a proper morphism $\pi:X\to S$ such that
	\begin{enumerate}
		\item $K_X+D_1+D_2+\Delta\sim_{\Rr,\pi}0$,
		\item $(X,D_1+D_2+\Delta)$ is plt,
		\item $\pi:D_i\to S$ are both isomorphisms, and
		\item every reduced fiber $\red X_s$ is isomorphic to $\Pp^1$.
	\end{enumerate}
	By definition, $\Delta$ is vertical and the projection gives an isomorphism of klt pairs
	$$(D_1,\Diff_{D_1}\Delta)\simeq (D_2,\Diff_{D_2}\Delta).$$
	A simple example of a standard $\Pp^1$-link is a product
	$$(S\times \Pp^1,S\times\{0\}+S\times\{\infty\}+\Delta_s\times\Pp^1)\to S.$$
	
\end{definition}
We will need the following result from \cite{KollarMMP13}, for a slightly weaker version, see Proposition 5.1 in \cite{kk10}.

\begin{theorem}[{\cite[Proposition 4.37]{KollarMMP13}}]\label{thm:shoconnkollar}
	Let $(X,B)$ be a klt pair, $C$ an effective $\Zz$-divisor on $X$, and $f: X\to Z$ a proper morphism such that $K_X+B+C\sim_{\Rr,f}0$. Fix a point $z\in Z$, and assume that $f^{-1}(z)$ is connected but $f^{-1}(z)\cap C$ is disconnected (as $k(z)$-schemes).
	
	\begin{enumerate}
		\item There is an \'{e}tale morphism $(z'\in Z')\to(z\in Z)$ and a proper morphism $T'\to Z'$ such that $k(z)=k(z')$ and $(X,B+C)\times_{Z}Z'$ is birational to a standard $\Pp^1$-link over $T'$.
		
		\item $(X,B+C)$ is plt in a neighborhood of $f^{-1}(z)$.
		\item $f^{-1}(z)\cap C$ has 2 connected components, and there is a crepant birational involution of $(C,\Diff_{C}B)$ that interchanges these two components.
	\end{enumerate}
\end{theorem}

\begin{definition}[$\Pp^1$-linking{\cite[Definition 4.39]{KollarMMP13}}] Let $g:(X,\Delta)\to S$ be a dlt log pair, $g$ a proper surjective morphism with connected fibers, $K_X+\Delta \sim_{\Rr,g}0$ and $Z_1,Z_2\subset X$ two lc centers. We say that $Z_1,Z_2$ are directly $\Pp^1$-linked if there is an lc center $W\subset X$ containing the $Z_i$ such that $g(W)=g(Z_1)=g(Z_2)$ and $(W,\Diff_{W}^{*}\Delta)$ is birational to a standard $\Pp^1$-link with $Z_i$ mapping to $D_i$.
	
	We say that $Z_1,Z_2\subset X$ are $\Pp^1$-linked if there is a sequence of lc centers $Z_1',\ldots,Z_m'$ such that $Z_1'=Z_1$, $Z_m'=Z_2$ and $Z_i'$ is directly $\Pp^1$-linked to $Z_{i+1}'$ for $i=1,\ldots,m-1$ (or $Z_1=Z_2$). Every $\Pp^1$-linking defines a birational map between $(Z_1,\Diff_{Z_1}^{*}\Delta)$ and $(Z_2,\Diff_{Z_2}^{*}\Delta)$.
	
	Note that $Z_1$ is a minimal lc center if and only if $(Z_1,\Diff_{Z_1}^{*}\Delta)$ is klt. In this case $Z_2$ is also a minimal log center.
	
\end{definition}

In order to prove Theorem \ref{thm:shoconn}, we need the following result.
\begin{theorem}[{\cite[Lemma 3.1]{AHK07}}]\label{AHK0631}
	Let
	$$(X,D)=(X^0,D^0)\dashrightarrow (X^{1},D^{1})\dashrightarrow(X^{2},D^{2})\dashrightarrow \cdots$$
	be a sequence of $4$-dimensional klt flips such that $D=E+D'$ with effective $E$ and $D'$, such that all flips are $E$-positive (i.e., E is positive on a curve in the flipping locus). Then the sequence of flips terminates.
\end{theorem}

\section{A strong connectedness theorem}
In this section, we will prove the following theorem, which is a slightly stronger result than Theorem \ref{thm:shoconn}.
\begin{theorem}\label{thm:shoconnstrong}
	Assume the termination of klt flips. Let $(X,\Delta)$ be a log pair, and $\pi: X\to S$ a proper morphism such that $-(K_X+\Delta)$ is $\pi$-nef. Fix a point $s\in S$, and assume that $\pi^{-1}(s)$ is connected but $\pi^{-1}(s)\cap \Nklt(X,\Delta)$ is disconnected (as $k(s)$-schemes).
	
	\begin{enumerate}
		\item There is an \'{e}tale morphism $(s'\in S')\to(s\in S)$ and a proper morphism $T'\to S'$ such that $k(s)=k(s')$ and $(X,\Delta)\times_{S}S'$ is birational to a standard $\Pp^1$-link over $T'$.
		
		\item $(X,\Delta)$ is plt in a neighborhood of $\pi^{-1}(s)$.
		\item $\pi^{-1}(s)\cap\Nklt(X,\Delta)$ has 2 connected components. 
	\end{enumerate}
\end{theorem}
\begin{proof}
	By Lemma \ref{lem:dlt}, we may take a dlt modification of $(X,\Delta)$, $\varphi: X_1\to X$, so that $X_1$ is $\Qq$-factorial, $K_{X_1}+B+C=\varphi^{*}(K_X+\Delta)$, $C=\lfloor B+C\rfloor$ and $(X_1,B+\red C)$ is dlt. By Lemma \ref{lem:crepant}, there is a bijection between the connected components of $\Nklt(X,\Delta)$ and $\Nklt(X_1,B+C)=\Supp C$. In particular, $\pi_1^{-1}(s)\cap \Supp C$ is disconnected, where $\pi_1:=\pi \circ \varphi$.  Arguing as in \cite[Proposition 4.37]{KollarMMP13}, we may assume that different components of $\pi_1^{-1}(s)\cap \Supp C$ belong to different components of $\Supp C$.

	If $-C$ is $\pi_1$-nef, then we claim that $\pi_1^{-1}(s)\subseteq \Supp C$, which is a contradiction. To see the claim, note that otherwise we can choose a curve $\Gamma \subset \pi_1^{-1}(s)$ such that $\Gamma  \not\subset \Supp C$ and $\Gamma \cap \Supp C\neq \emptyset$. But then $-C\cdot\Gamma <0$ which is impossible.

	Therefore we may assume that $-C$ is not $\pi_1$-nef. Let $B_1:=B$, $P:=-(K_{X_1}+B+C)$,
	$$\mu_1:=\max\{\mu|P-\mu C\text{ is }\pi_1 \text{-nef}\}\in [0,+\infty ),$$
	$P_1:=P-\mu_1 C$, and $C_1:=(1+\mu_1)C$. 
	
	Choose $\beta_1>1$ as in Lemma \ref{lem:extremal}, then
	$$K_{X_1}+B_1+\beta_1P_1=(\beta_1-1)P-(\beta_1\mu_1+1)C$$
	is not nef, since $\frac{\beta_1\mu_1+1}{\beta_1-1}>\mu_1$.
	
	We run a special kind of $(K_{X_1}+B_1)$-MMP over $S$ as follows. Let $R$ be a $(K_{X_1}+B_1+\beta_1P_1)$-negative extremal ray over $S$ such that $P_1\cdot R=0$. It is also a $(K_{X_1}+B_1)$-negative extremal ray over $S$. By the cone theorem, $R$ can be contracted. If we get a Mori fiber space, $X_1\to Z$, then we stop. Otherwise, we get a birational map, $\phi:X_1\dashrightarrow X_2$. By the cone theorem, $\phi_{*}P_1$ is nef and $(X_2,\phi_{*}B_1)$ is klt. Let $p:W\to X_1$ and $q:W\to X_{2}$ be a common resolution. Since $K_{X_1}+B_1+C_1=-P_1$ is $\phi$-trivial, according to the negativity lemma, $$p^{*}(K_{X_1}+B_1+C_1)=q^{*}(K_{X_{2}}+\phi_{*}(B_1+C_1)).$$ By Lemma \ref{lem:crepant}, there is a bijection between the connected components of $\Nklt(X_1,B_1+C_1)=\Supp C_1$ and $\Nklt(X_{2},\phi_{*}(B_1+C_1))=\Supp (\phi_{*}C_1)$. If $-\phi_{*}C_1$ is $\pi_{2}$-nef over $S$, where $\pi_2: X_2\to S$ is the induced morphism, then $\pi_2^{-1}(s)\subseteq \Supp \phi_{*}C_1$, a contradiction. Therefore $-\phi_{*}C_1$ is not $\pi_{2}$-nef over $S$ and we can repeat the above process, namely, let $B_2:=\phi_{*}B_1$,
	$$\mu_2:=\max\{\mu|\phi_{*}P-\mu \phi_{*}C\text{ is } \pi_2 \text{-nef}\}\in [\mu_1, +\infty),$$
	$P_2:=\phi_{*}(P-\mu_2 C)$, $C_2:=(1+\mu_2)\phi_{*}C$. Choose $\beta_2>1$ as in Lemma \ref{lem:extremal}. Arguing as above, one sees that $(K_{X_2}+B_2+\beta_2P_2)$ is not nef and we may contract a $(K_{X_2}+B_2+\beta_2P_2)$-negative extremal ray over $S$. Proceeding this way, we obtain a sequence of the $(K_{X_1}+B_1)$-MMP over $S$.
	
	Since we are assuming the termination of flips, this MMP $\psi:X_1\dashrightarrow X_n$ terminates with a Mori fiber space $X_n\to Z$. For simplicity, we let $Y:=X_n$, $$B_Y:=B_{n}=\psi _*B,\, C_Y:=C_{n}=(1+\mu _n)\psi _* C,\, P_Y:=P_n=\psi _*(P-\mu _nC).$$ We remark that since $\gamma:Y\to Z$ is a Mori fiber space, its fibers are connected, and the fiber $\rho^{-1}(s)$ is connected, where $\rho:Z\to S$ is the natural induced morphism. Now, $-(K_Y+B_Y)=P_Y+C_Y$ is ample over $Z$, or equivalently, $C_Y$ is ample over $Z$. This implies that $C_Y$ intersects every fiber of $\gamma$. As observed  above, there is a bijection between the connected components of $\Nklt(X_1,B_1+C_1)=\Supp C_1$ and $\Nklt(Y,B_Y+C_Y)=\Supp C_Y$, and $\gamma^{-1}(\rho^{-1}(s))\cap \Supp C_Y$ is disconnected.
	
	$$
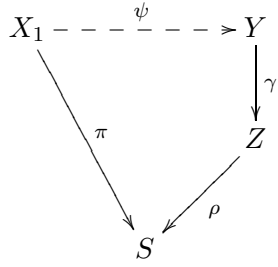
\xymatrix@=2.5em{
		X_1\ar@{-->}[rr]^{\psi} \ar[ddr]^{\pi}  &          &  Y\ar[d]^{\gamma}\\
		&         &  Z\ar[dl]^{\rho}\\
		& S      &
	}
	$$
	
	If $\mu_n>0$, then $-(K_Y+B_Y+\psi_{*}C)=(P_Y+\mu_n\psi_{*}C)$ is ample over $Z$. By the Shokurov-Koll\'{a}r connectedness lemma, $\Nklt(Y,B_Y+\psi_{*}C)=\Supp C_Y$ is connected in a neighbourhood of any fiber of $\gamma$. Hence,
	$$\gamma^{-1}(\rho^{-1}(s))\cap \Nklt(Y,B_Y+\psi_{*}C)=\gamma^{-1}(\rho^{-1}(s))\cap \Supp C_Y$$
	is connected, a contradiction.
	
	If $\mu_n=0$, then $\mu_1=0$ and $C_Y=\psi_{*}C$ since $\mu_n\ge\mu_{n-1}\ge\cdots\ge\mu_1\ge0$. The MMP is $P$-trivial and hence $(K_X+B+C)$-trivial, and so $K_Y+B_Y+C_Y$ is numerically trivial over $Z$. 
	
	Since
	$$\gamma^{-1}(\rho^{-1}(s))\cap\Nklt(Y,B_Y+C_Y)=\gamma^{-1}(\rho^{-1}(s))\cap\Supp C_Y$$
	is disconnected and $\rho^{-1}(s)$ is connected, 
	there exists $z_0\in \rho^{-1}(s)$, such that $\gamma^{-1}(z_0)\cap \Supp C_Y$ is disconnected.
	
	Since $C_Y$ is ample over $Z$, there is an irreducible component, say $D_1\subset C_Y$ that has positive intersection with the contracted ray. Since $\rho (Y/Z)=1$, if a vertical divisor intersects  a fiber, then this divisor must contain the fiber. Since $\gamma^{-1}(z_0)\cap \Supp C_Y$ is disconnected, there exists another prime divisor, say $D_2\subset C_Y$ that intersects with the fiber $\gamma^{-1}(z_0)$ and is disjoint from $D_1$ in $\gamma^{-1}(z_0)$. Thus, $D_2$ also has positive intersection with the contracted ray, and is ample over $Z$. $D_1$ and $D_2$ both intersect every curve contracted by $\gamma$ and they are disjoint in $\gamma^{-1}(z_0)$. Hence, a general fiber $F_g$ of $\gamma$ is a smooth rational curve. Now, $(F_g\cdot(B_Y+C_Y))=-\deg K_{F_g}=2$. Possibly shrinking $S$, we see that $D_1$ and $D_2$ are sections of $\gamma$, they both appear in $C_Y$ with coefficient 1 and $B_Y+C_Y-D_1-D_2$ consists of vertical divisors, and the vertical components of $C_Y$ are disjoint from $\gamma^{-1}(\rho^{-1}(s))$. By shrinking $S$ again, we may assume that $C_Y=D_1+D_2$, and for any $z\in \rho^{-1}(s)$, $\gamma^{-1}(z)$ is disconnected. 
	
	
	By applying Theorem \ref{thm:shoconnkollar} to any $z\in \rho^{-1}(s)$, we get Theorem \ref{thm:shoconnstrong}(1), $(Y,B_Y+C_Y)$ is plt in a neighborhood of $\gamma^{-1}(\rho^{-1}(s))$, and $\pi_{1}^{-1}(s)\cap\Supp C_Y$ has two connected components. Since $\psi:(X_1,B_1+C_1)\dashrightarrow (Y,B_Y+C_Y)$ is a crepant birational map, and none of the irreducible components of $\Supp C_Y$ are contracted by $\psi^{-1}$, $(X_1,B_1+C_1)$ is plt in a neighborhood of $\pi_{1}^{-1}(s)$ (cf. \cite[2.23.4]{KollarMMP13}). Therefore, $(X,\Delta)$ is also plt in a neighborhood of $\pi^{-1}(s)$.
\end{proof}

\section{Proof of the main results}
\begin{proof}[Proof of Theorem \ref{thm:shoconn}]
	Assuming the termination of klt flips, this follows from Theorem \ref{thm:shoconnstrong}.
	
	Assume now that $\dim X\leq 4$.
	We follow the notation in the proof of Theorem \ref{thm:shoconnstrong}. We only need to prove that the sequence of $(K_{X_1}+B_1)$-flips terminates. Since $(X_1,B_1)$ is klt, there exists $M>1$, such that $(X_1,B_1+\frac{C}{M})$ is klt. Each step of the MMP is $(P-\mu_iC)$-trivial for some $\mu_i\ge 0$, or equivalently $(K_{X_1}+B_1+(1+\mu_i)C)$-trivial. It follows that the sequence of the $(K_{X_1}+B_1)$-MMP is also a sequence of the $(K_{X_1}+B_1+\frac{C}{M})$-MMP, and all steps are $C$-positive. Hence, by Lemma \ref{AHK0631}, the sequence of flips terminates.
\end{proof}

\begin{proof}[Proof of Corollary \ref{cor:shoconn}]
	According to the proof of Theorem \ref{thm:shoconnstrong}, if $\mu_1>0$, then $\Nklt(X,\Delta)$ is connected in a neighbourhood of any fiber of $\pi$.   The nefness of $-(K_X+B+(1+\epsilon)C)$ implies $\mu_1>0$. 
\end{proof}
We remark that the result of Corollary \ref{cor:shoconn} without assuming the termination of flips implies the connectedness lemma. For a log pair $(X,\Delta)$ and a proper morphism $\pi: X\to S$, such that $-(K_X+\Delta)$ is $\pi$-nef and $\pi$-big. Possibly by taking a dlt modification, we may assume that $(X,B+\red C)$ is dlt, where $B=\{\Delta\}$ and $C=\lfloor \Delta\rfloor$. There exists an effective divisor $E$ and a $\pi$-ample divisor $A_k$, such that $-(K_X+\Delta)=A_k+\frac{E}{k}$ for $k\gg 1$. Let $\Delta_k:=\Delta+\frac{E}{k}$, $B_k:=\{\Delta_k\}$ and $C_k:=\lfloor \Delta_k\rfloor$. There exists $k\gg 1$, such that $(X,B_k+\red C_k)$ is dlt and $\Nklt(X,\Delta)=\Nklt(X,\Delta_k)$. Since $-(K_X+\Delta_k)$ is ample, $-(K_X+B_k+(1+\epsilon)C_k)$ is nef for some $\epsilon>0$. By Corollary \ref{cor:shoconn}, $\Nklt(X,\Delta_k)$ is connected in a neighbourhood of any fiber of $\pi$.

\begin{proof}[Proof of Theorem \ref{thm:soucelccenter}]
	This follows from Theorem \ref{thm:shoconnstrong} and the proof of Theorem 4.40 in \cite{KollarMMP13}.
\end{proof}

\begin{proof}[Proof of Corollary \ref{cor:soucelccenter}]
	This follows from  Theorem \ref{thm:soucelccenter} and the proof Corollary 4.41 in \cite{KollarMMP13}.
\end{proof}

It would be interesting to prove the results in this note without assuming the termination of flips, and to generalize these results to the setting of generalized polarized pairs.

\bibliographystyle{abbrv}

\end{document}